\documentclass{amsart}
\usepackage{amssymb, hyperref}
\usepackage{color}

\newtheorem{thm}{Theorem}
\newtheorem{prp}{Proposition}

\newtheorem{lma}{Lemma}

\theoremstyle{definition}
\newtheorem{df}{Definition}
\newtheorem{nt}{Notation}

\newtheorem{ex}{Example}

\def\dist{{\rm dist}}

\begin{document}
\title{Dirichlet problems on graphs with ends}

\author{Tony L. Perkins }
\address{ Department of Mathematics,  Spring Hill College, 4000 Dauphin Street, Mobile, Alabama 36608-1791}
\date{\today}

\begin{abstract}
In classical potential theory, one can solve the Dirichlet problem on unbounded domains such as the upper half plane.  These domains have two types of boundary points; the usual finite boundary points and another point at infinity.  W. Woess has solved a discrete version of the Dirichlet problem on the ends of graphs analogous to having multiple points at infinity and no finite boundary. Whereas C. Kiselman has solved a similar version of the Dirichlet problem on graphs analogous to bounded domains.  In this work, we combine the two ideas to solve a version of the Dirichlet problem on graphs with finitely many ends and boundary points of the Kiselman type.
\end{abstract}

\maketitle
\thispagestyle{empty}

\section{Introduction}
Let $D\subset\mathbb{R}^n$ be an open set and take a function $f\colon \partial D\rightarrow \mathbb{R}$.  The Dirichlet problem on $D$ with boundary data $f$ is to find a unique function $h$ which is harmonic on $D$, continuous on $\overline{D}$, and agrees with $f$ on the boundary, i.e. $h|_{\partial D}=f$.  Various methods, such as that of Perron, allow one to always find a harmonic function $h$ which is associated to $f$ in a natural way.  The main challenge is to determine when (or where) this $h$ matches $f$ on $\partial D$.  Given the way we've defined the Dirichlet problem the continuity of $f$ is clearly necessary.  For bounded open sets $D$, the uniqueness criteria follows from the maximum principle for harmonic functions.  However for unbounded domains uniqueness is non-trivial.

Consider the example of the upper half plane $\mathbb{H}=\{(x,y)\in \mathbb{R}^2\colon y\ge 0\}$.  If one wants to solve the Dirichlet problem on $\mathbb{H}$, we have trouble with uniqueness.  Indeed, the functions $f_1(x,y)=y$ and $f_2(x,y)=0$ are harmonic and identically equal to $0$ on the boundary of the upper half plane. In classical potential theory, to obtain uniqueness one adds a point at infinity, $\hat{\mathbb{H}}=\mathbb{H}\cup \{\infty\}$, and then considers the Dirichlet problem on $\hat{\mathbb{H}}$ with boundary $\partial\hat{\mathbb{H}}=\partial\mathbb{H}\cup\{\infty\}$.
Using Perron's method, to each $f\colon \partial \hat{\mathbb{H}}\rightarrow \mathbb{R}$, we can find a corresponding function $h\colon \hat{\mathbb{H}}\rightarrow \mathbb{R}$ which is harmonic.  To ensure that this $h$ equals $f$ on $\partial \hat{\mathbb{H}}$ all that is required of $f$ is to be continuous and bounded on $\partial \hat{\mathbb{H}}$, where continuous on $\partial\hat{\mathbb{H}}$ means that $f$ is continuous on $\partial \mathbb{H}$ and continuous at the point infinity.

In discrete potential theory, one often studies the Dirichlet problem on directed graphs.  Here one considers an arbitrary at most countable set $X$ as the vertex set.  Weighted directed edges are provided by a structure function $\lambda:X\times X \rightarrow \mathbb{R}$, where $\lambda\ge 0$, and $\lambda(x,y)>0$ denotes the existence of a directed edge from $x$ to $y$ with edge weight equal to $\lambda(x,y)$  ($\lambda(x,y)=0$ implies that there is no edge from $x$ to $y$), and with the additional condition $\sum_{y\in X} \lambda(x,y)=1$, the graph has the structure of a Markov chain.

Considering the domain as a Markov chain allows one to attack these types of problems with either probability or analysis.  This is of course true for the classical setting as well, e.g. \cite{D01}.  Here we take the analytic approach.  

In \cite{K05} Kiselman defines the boundary of $X$ as
\[\partial{X} = \{x_0\in X\colon \lambda(x_0,x_0)=1\}\equiv \{x_0\in X \colon \lambda(x_0,y)=0 \text{ for all } y\in X\setminus\{x_0\}\}\]
and studies the Dirichlet problem in this setting.  To obtain uniqueness results, he considers only finite graphs.

However in probability one often studies a {\it reversible} Markov chain, i.e. $\lambda(x,y)>0 \iff \lambda(y,x)>0$.  In this setting Kiselman type boundary points trivially become isolated points of the graph.  Thus the only interesting formulation of a Dirichlet problem is a problem at infinity.  The principle works in this setting are those of \cite{CSW93,W88, W96}, where they solve a Dirichlet problem where the boundary points are the ends of the graph.

The primary goal of this paper is to begin the study of discrete Dirichlet problems with mixed boundary type, i.e. a graph with both Kiselman type boundary points and ends.  As a motivational example we look at the following simple `discretized' version of the upper half plane.
\begin{ex}
Consider the set $X=\mathbb{Z}^2 \cap \mathbb{H}$.  Define
\[\lambda((x_1,y_1), (x_2,y_2)) =
\begin{cases}
1/4 &\colon  (x_2-x_1)^2 + (y_2-y_1)^2=1 \text{ and } y_1 \neq 0 \\
1 &\colon x_1=x_2, y_1=y_2 \text{ and } y_1=0 \\
0 &\colon \text{ otherwise.}
\end{cases}\]
Then $\partial X = \mathbb{Z}^2 \cap \{y=0\}$ is a boundary set.  It is easy to check that the functions $f_1(x,y)=y$ and $f_2(x,y)=0$ are harmonic and identically equal to $0$ on $\partial X$.
\end{ex}

So uniqueness remains of principal concern.  However in $\hat{X}$, the end compactification of $X$, we add a single point at infinity.  Therefore $\hat{X}$ has the interesting structure of both a Kiselman type boundary and one end.

To study the Dirichlet problem on graphs of this type we introduce the notion of a quasi-reversible graph (reversible except for some Kiselman type boundary points).  We prove (Theorem \ref{T:max-prnc}) a maximum principle with regard to a boundary of mixed type, from which follows (Theorem \ref{T:unique}) a uniqueness result that holds for all connected quasi-reversible graphs.

It is a remarkable property of these mixed boundary problems that one-ended graphs are already non-trivial; a stark contrast to the setting in \cite{CSW93, W96}.  In the last section, we show (Theorem \ref{T:exists}) there exists a unique solution to Dirichlet problems of mixed boundary types on one ended graphs.  Finally we are able to extend the previous result to show (Theorem \ref{T:exists2}) there also exists a unique solution on graphs with finitely many ends.  Analogously to the classical setting we require the boundary data to be continuous at infinity.  

The Dirichlet problem, i.e.\ harmonic interpolation, on graphs has applications to coverage problems on topological sensor networks, \cite[pp 62--63]{G10} and shape description problems in digital image analysis \cite{K05, W05}.

\section{Preliminaries}\label{S:basic}
Introductions to various aspects of discrete potential theory can be found in \cite{BLS07, S94, W94}. The basic structure we will be working over can be thought of as a directed graph with weights on the edges, which is also known as a network.  The weights will be assumed to be non-negative and can be thought of as transition probabilities.  Therefore another interpretation, e.g. \cite{W94}, is as a random walk on a graph or a Markov chain.

We start by considering an arbitrary at most countable set $X$.  This is the domain of our harmonic functions and can be thought of as the vertices of a graph.

A function $f:X\times X\rightarrow \mathbb{C}$ is called a \emph{structure function} if the set $\{y\colon f(x,y)\neq 0\}$ is finite for all $x\in X$, thereby ensuring that our graph is locally finite.  A common operation one takes on structure functions is given by
\[(f\diamond g)(x,y) := \sum_{\zeta \in X} f(x,\zeta)g(\zeta,y) \qquad (x,y\in X). \]
In fact we are interested in a particular type of structure function called a \emph{weight function} which has the additional properties $\lambda \colon X\times X \rightarrow \mathbb{R}$, with $\lambda \ge 0$ and $\sum_{\zeta\in X} \lambda(x,\zeta) = 1$ for all $x\in X$.

We think of $\lambda(x,y)$ as defining a weight on a directed edge from $x$ to $y$ with the interpretation of $\lambda(x,y)=0$ as meaning that there is no edge.  For us the domain $X$ and the weight $\lambda$ are fixed, hence we will often refer to the graph $X$ without ambiguity. Also note that $\lambda$ need not be symmetric.

\subsection{The Laplacian and harmonic functions}\label{SS:lap_sub}

The Laplacian $\Delta$ at a point $x\in X$ is defined with respect to $\lambda$ on functions $f\colon X \rightarrow \mathbb{R}$ by
\[\Delta f(x) = \sum_{\zeta \in X} \lambda(x,\zeta)[f(\zeta)-f(x)].\]

A function $f\colon X \rightarrow \mathbb{R}$ is said to be \emph{harmonic} if $\Delta f(x)=0$ for every $x\in X$ and \emph{subharmonic} if $\Delta f(x)\ge 0$ for every $x\in X$.  Some simple arithmetic shows that subharmonicity is equivalent to the property that
\[f(x) \le \sum_{\zeta \in X} \lambda(x,\zeta)f(\zeta) \qquad (x\in X).\]

Notice that constant functions are always harmonic, no matter what weight $\lambda$ is under consideration.
\begin{nt}
The set of harmonic and subharmonic functions on $X$ will be denoted by $H(X)$ and $S(X)$, respectively.
\end{nt}

Some properties of subharmonic functions and their relationship to other functions on graphs can be found \cite{BP14}.

\subsection{Kiselman's Boundary}\label{SS:Kiselman}
In \cite{K05}, Kiselman introduces the following notion of boundary on the directed graph $X$.  The boundary of $X$, denoted by $\partial X$, is the set
\begin{align*}
\partial X &= \{x\in X\colon \lambda(x,y)=0 \text{ for all } y \in X\setminus\{x\} \}\\
&= \{x\in X\colon \lambda(x,x)=1 \}.
\end{align*}
The Kiselman boundary can be thought of intuitively as the stable equilibria of the Markov process, those states from which a random walk can not escape.

Its complement $X^\circ = X \setminus \partial X$ is the \emph{interior} of $X$. We note that a point $x$ has a neighbor different from $x$ if and only if $x$ is in the interior. Given a point $a \in X$ we define $N_0(a) = \{a\}$ and then inductively
\[N_{k+1}(a) = \{y\colon \lambda(x, y) > 0 \text{ for some } x \in N_k(a)\}\qquad (k \in \mathbb{N}).\]
The union of all the $N_k(a)$ will be called the \emph{component} of $a$ and be denoted by $C(a)$. Clearly $C(a) = \{a\}$ if and only if $a$ is a boundary point.  We shall say that $X$ is \emph{boundary connected} if $C(a)$ intersects $\partial X$ for every $a \in X^\circ$. (In particular $\partial X$ is nonempty if $X$ is boundary connected and
nonempty.) We shall say that $X$ is \emph{connected} if $C(a) = X$ for all $a \in X^\circ$. It is possible that $\partial X$ may be empty.

It is useful to observe that for any function $f$ on $X$ we have
\begin{equation*}\sum_{\zeta \in X} \lambda(x , \zeta) f(\zeta) = \sum_{\zeta \in C(x)} \lambda(x , \zeta) f(\zeta)=\sum_{\zeta \in N_1(x)} \lambda(x , \zeta) f(\zeta),
\end{equation*}
as $\lambda(x, \eta)=0$ for all $\eta \in X\setminus N_1(x)$ by definition.

Kiselman (see \cite{K05}) provides some maximum principles for subharmonic functions on directed graphs, which can be traced back to \cite{LT69}.
\begin{prp}[Kiselman]\label{P:max1}
If $X$ is finite and boundary connected, then $\sup_{X}u = \sup_{\partial X}u$ for all subharmonic functions $u$ on $X$. The converse holds.
\end{prp}

Later in Section \ref{S:max_nn-cnst}, Proposition \ref{P:max1} of Kiselman is generalized by our Theorem \ref{T:max-prnc} to include the ends of a graph.

\begin{prp}[Kiselman]\label{P:max2}
If $X$ is connected (finite or infinite), then a subharmonic function which attains its supremum at an interior point must be constant.  The converse holds.
\end{prp}

\begin{nt}
For $k= 1, 2, \cdots$, we take $\lambda^{\diamond k}$ as the $\diamond$ product of $\lambda$ $k$ times, i.e. $\lambda^{\diamond k} = \lambda \diamond \lambda\diamond \lambda\cdots \diamond \lambda$.

Take $\omega_x^0(\zeta)= \delta(x,\zeta)$, $\omega_x^1(\zeta)=\lambda(x,\zeta)$, and $\omega_x^k(\zeta)=\lambda^{\diamond k}(x,\zeta)$.  In a slight abuse of notation for any function $g$ on $X$, we let $\omega_x^k(g)$ denote
\[\omega_x^k(g) := \sum_{\zeta \in X}\lambda^{\diamond k}(x,\zeta)g(\zeta).\]
\end{nt}

We see how this notation is useful in the following technical lemma.
\begin{lma}\label{L:mono}
If $g$ is a subharmonic function on $X$ then $\{\omega_x^k(g)\}_{k=0}^\infty$ is an increasing sequence, i.e. $\omega_x^k(g) \le \omega_x^{k+1}(g)$ for all $x\in X$ and all $k\in \mathbb{N}$.
\end{lma}
\begin{proof}
First observe that
\[\omega_x^0(g)=g(x)\le \sum_{\zeta\in X} \lambda(x,\zeta) g(\zeta) = \omega_x^1(g),\]
holds for any subharmonic $g$ and any $x\in X$ by the definitions. We may repeat the substitution at $\zeta$ inside the summation as
\[g(x)\le \sum_{\zeta\in X} \lambda(x,\zeta) g(\zeta)\le \sum_{\zeta\in X} \lambda(x,\zeta)\left(\sum_{\eta\in X} \lambda(\zeta,\eta) g(\eta)\right).\]
The summations are all finite by hypothesis (the graph is locally finite) and so we swap the order of summation
\[\omega_x^0(g)\le\omega_x^1(g) \le \sum_{\eta\in X}\left(\sum_{\zeta\in X} \lambda(x,\zeta) \lambda(\zeta,\eta)\right) g(\eta) = \sum_{\eta\in X}(\lambda\diamond \lambda)(x,\eta) g(\eta) = \omega_x^2(g).\]
Clearly this process may be repeated.  Therefore $\{\omega_x^k(g)\}_{k=0}^\infty$ is an increasing sequence.
\end{proof}

\subsection{The ends of a graph}

We define a {\it ray} in  an undirected infinite graph $X$ as a one way infinite cycle-free (undirected) path $\pi = x_0x_1x_2\ldots$ in $X$. For directed graphs, we will require rays to be directed; that is, $\lambda(x_i, x_{i+1}) > 0$ for all $i \geq 0$.

Following \cite{W96}, consider any finite subgraph $A$ of $X$.  Then $X\setminus A$ breaks into finitely many connected components, one of which must contain infinitely many of the points of $\pi$.  In this case we say $\pi$ is said to {\it end} in that component.  Given two rays $\pi$ and $\pi'$ we say that they are {\it equivalent} if $\pi$ and $\pi'$ end in the same component of $X\setminus A$ for every finite $A\subset X$.  An {\it end} of a graph is precisely one such equivalence class of rays.

We let $\partial_{end}X$ denote the collection of ends of $X$ and take $\hat{X}=\partial_{end}X \cup X$.  The boundary of $\hat{X}$ is given by $\partial \hat{X}=\partial_{end}X \cup \partial X$ where
\[\partial X = \{x\in X\colon \lambda(x,x)=1\}.\]  This is the boundary from which we wish to extend harmonic functions.  First we will want to discuss the class of subharmonic functions on $\hat{X}$.
\begin{df}
The set of \emph{subharmonic functions} on $\hat{X}$ is denoted $S(\hat{X})$. A function $f\colon\hat{X}\rightarrow \mathbb{R}$ is said to be in $S(\hat{X})$ if $f|_X\in S(X)$ and
\[ \limsup_{i\rightarrow \infty}f(x_i) \le f(\omega),\]
where $\pi = x_0x_1x_2\ldots $ is in $\omega$.
\end{df}
Of course, each subharmonic function $f$ on $X$ naturally extends to $\hat{X}$ in the following way:
\[f(\omega) = \sup\{ \limsup_{i\rightarrow \infty}f(x_i)\colon \pi = x_0x_1x_2\ldots  \text{ in } \omega \}.\]
Consequently, the class of harmonic functions on $\hat{X}$ is \[H(\hat{X}) = S(\hat{X})\cap (-S(\hat{X})),\]
where $-S(\hat{X})=\{-f\colon f\in S(\hat{X})\}$.

The Dirichlet problem on $\hat{X}$ can now be phrased: for any $f\colon \partial \hat{X} \rightarrow \mathbb{R}$ there is a unique $h\in H(\hat{X})$ such that $h|_{\partial\hat{X}}=f$.

\section{Maximum principles and uniqueness for the Dirichlet problem.}\label{S:max_nn-cnst}

In this section we prove a maximum principle for a class of graphs.  This will be used to supply the ``uniqueness" part of the Dirichlet problem.

\begin{df}
Given a non-constant subharmonic function $f$, we call a finite sequence $\{x_i\}_{i=0}^k \subset X$ such that $f(x_i)< f(x_{i+1})$, $x_{i+1}\in N_1(x_i)$ and $f(x_{i+1}) = \max\{f(y)\colon y\in N_1(x_i)\}$ a \emph{maximally increasing path} for $f$. Note that such a path is cycle-free. Consequently, we define {\em a maximally increasing ray} for $f$ to be an infinite sequence $\{x_i\} \subset X$ such that every finite subsequence is a maximally increasing path for $f$.
\end{df}

\begin{lma}[K{\"o}nig's Infinity Lemma] Let $V_0, V_1, \ldots$ be an infinite sequence of disjoint non-empty finite sets, and let $X$ be an undirected graph on their union. Assume that every vertex $v \in V_n$ has a neighbor $g(v)$ in $V_{n-1}$. Then $X$ contains a ray $v_0v_1\ldots$ with $v_n \in V_n$ for all $n$.
\end{lma}

\begin{thm}\label{T:max_incr}
Suppose $(X,\lambda)$ is a reversible infinite connected graph.  Given any non-constant subharmonic function $f$ there exists a maximally increasing ray for $f$.
\end{thm}
\begin{proof}
Suppose $f$ is a non-constant subharmonic function.  Then as $X$ is connected there exists two points $a,b\in X$ such $a\neq b$, $\lambda(a,b)>0$, consequently $\lambda(b,a)>0$, and $f(a)< f(b)$.  Take $x_0=a$ and $x_1\in X$ such that $f(x_1) = \max\{f(y)\colon y\in N_1(x_0)\}$.  As $b\in N_1(x_0)$ we have that $f(x_0)<f(b)\le f(x_1)$.

We proceed to built the sequence inductively.  Suppose we have a maximally increasing path $\{x_i\}_{i=0}^k$ with $f(x_i)< f(x_{i+1})$, $x_{i+1}\in N_1(x_i)$ and $f(x_{i+1}) = \max\{f(y)\colon y\in N_1(x_i)\}$.  We will now find the next term.  Since $\lambda\ge 0$ and $\sum_{y\in X}\lambda(x_k,y)=1$, we think of subharmonicity as averaging the values of $f$ over the nearest neighbors with weights $\lambda$.  By construction we have that $\lambda(x_{k-1}, x_k)>0$.  Hence the reversibly condition implies that the average $\sum_{y\in X}\lambda(x_k,y)f(y)$ includes the $x_{k-1}$ term as $\lambda(x_k,x_{k-1})>0$.  Since $f(x_{k-1})<f(x_k)$ and $f$ is subharmonic, there must be a point $c\in N_1(x_k)$ such that $f(x_k)<f(c)$.  Take $x_{k+1}$ to be a point in $N_1(x_k)$ with $f(x_{k+1})=\max\{f(y)\colon y\in N_1(x_k)\}$ and we have that $f(x_k)<f(c)\le f(x_{k+1})$. This guarantees that $x_{i+1}$ is distinct from $\{x_i\}_{i=0}^k$, so $\{x_i\}_{i=0}^{k+1}$ is a maximally increasing (cycle-free) path. Whence, given any maximally increasing path of length $k$, we can always extend it to a maximally increasing path of length $k+1$. By induction and K{\"o}nig's Infinity Lemma, $X$ contains a maximally increasing ray for $f$.
\end{proof}

Recall that a graph is reversible whenever $\lambda(x,y)>0 \iff \lambda(y,x)>0$. With the next definition we take a weaker form of reversible by allowing some Kiselman type boundary points.
\begin{df}
We say that a graph is {\it quasi-reversible} if $\lambda(x,y)>0$ implies that either $\lambda(y,x)>0$ or $\lambda(y,y)=1$, that is, the graph is reversible everywhere except the boundary points.
\end{df}
Observe that Example 1 is of this type.

\begin{thm}[Maximum Principle]\label{T:max-prnc}
Let $X$ be a connected quasi-regular graph.  Then  for all $f\in S(\hat{X})$ we have
\[\sup\{f(x)\colon x\in \hat{X}\} = \sup\{f(y)\colon y\in \partial \hat{X}\}.\]
\end{thm}
\begin{proof}
Let $M=\sup\{f(x)\colon x\in \hat{X}\}$.

We start with the easy case when the supremum occurs at a point. If there is a point $x_0\in \partial \hat{X}$ such that $f(x_0)=M$, we are done.  Suppose otherwise, that there is a point $x_0\in \hat{X}^\circ=X^\circ$, so that $f(x_0)=M$.  Consequently by Proposition \ref{P:max2}, the function $f$ must therefore be constant on $X$ as the graph is connected, which trivially implies the result.

Now we consider what happens when the supremum does not occur at a point.  In this case, there is a sequence $\{x_i\}\subset \hat{X}$ such that $f(x_i)<f(x_{i+1})$ with $\lim_{i\rightarrow \infty}f(x_i)=M$.  If $|\{x_i\}\cap \partial\hat{X}|=\infty$, we are done.  Hence without loss of generality we will assume that $\{x_i\}\subset \hat{X}^\circ$.

Our strategy will be to start a directed path at each $x_i$ (or some other point $y$ where $f(y)=f(x_i)$) on which $f$ is strictly increasing.  This path will either terminate at a boundary point of $X$ or be a ray (an infinite cycle free one way directed path) on which $f$ is strictly increasing.  We will then use this path to conclude the result.

To construct such a path we mimic the previous Theorem \ref{T:max_incr}.   Pick any $x_i$ from $\{x_i\}$.  It may be the case that $f$ is locally constant around $x_i$, that is, $f(x_i)=f(y)$ for all $y$ adjacent to $x_i$.  However since $f$ is non-constant and we are working on a connected quasi-reversible graph, we can find a $y_0$ such that $f(x_i)=f(y_0)$ such that $f$ is not locally constant around $y_0$.  Hence there is a $z$ adjacent to $y_0$ such that $f(z)\neq f(y_0)$.  If $f(z)>f(y_0)$, then take $y_1=z$.  If $f(z)<f(y_0)$, then subharmonicity demands the existence of a $z'$ adjacent to $y_0$ with $f(z')>f(y_0)$; in this case take $y_1=z'$.

If $y_1 \not\in\partial X$, we can find (using subharmonicity and that $f(y_0)<f(y_1)$) a $y_2$ adjacent to $y_1$ such that $f(y_2)>f(y_1)$; otherwise we allow the path to terminate at the point $y_1\in \partial X$.  Using induction we obtain a cycle-free directed path starting at $y_0$ (where $f(x_i)=f(y_0)$) along which $f$ is strictly increasing where the path is either $(1)$ finite with terminal point in $\partial X$ or $(2)$ a ray.

In the first case, we have the path $y_0y_1 \ldots y_m$ where $y_m\in \partial X$ and $f(y_i)<f(y_{i+1})$ and where $\lambda(y_i, y_{i+1})>0$.  Hence
\[f(x_i)=f(y_0)\le f(y_m) \le \sup\{f(y) \colon y\in \partial X\}\le \sup\{f(y) \colon y\in \partial \hat{X}\}.\]

While in the second case, we've constructed a ray $\pi = y_0y_1 \ldots $ with $f(y_i)<f(y_{i+1})$ where $\lambda(y_i, y_{i+1})>0$ and $f(x_i)=f(y_0)$.  Therefore $\pi$ must be in some end $\omega\in\partial_{end}X$, and so
\[f(x_i)=f(y_0)\le \lim_{y_i\in \pi}f(y_i) \le f(\omega) \le \sup\{f(y) \colon y\in \partial \hat{X}\}.\]

Thus in either case we have that $f(x_i)\le \sup\{f(y) \colon y\in \partial \hat{X}\}$. We simply take the supremum over all $x_i \in \{x_i\}$ to see
\[M = \sup\{f(x_i)\colon x_i\in\{x_i\} \} \le \sup\{f(y) \colon y\in \partial \hat{X}\}\]
which is the desired result.
\end{proof}

\begin{thm}[Uniqueness]\label{T:unique}
Let $X$ be connected and quasi-reversible.  If there exists a solution to the Dirichlet problem on $\hat{X}$, then it is unique.
\end{thm}
\begin{proof}
Let $f\colon \partial\hat{X}\rightarrow \mathbb{R}$.  We seek to show that there is at most one solution to the Dirichlet problem, that is a harmonic $h$ such that $h|_{\partial\hat{X}}=f$.

Suppose $h$ and $h'$ are harmonic and equal to $f$ on $\partial\hat{X}$.  Then $h-h'$ and $h'-h$ are subharmonic and equal to $0$ on $\partial\hat{X}$.  Hence by Theorem \ref{T:max-prnc}
\[\sup\{h(x)-h'(x)\colon x\in \hat{X}\} = \sup\{h'(x)-h(x)\colon x\in \hat{X}\}= \sup\{0\colon y\in \partial\hat{X}\}=0.\]
Therefore $h-h'\le 0$ and $h'-h\le 0$ on $\hat{X}$, and so they are equal.
\end{proof}

\section{Existence for graphs with one end.}
The following trivial observation will be useful.
\begin{lma}\label{L:ext_by_zero}
If $f\colon \partial X \rightarrow \mathbb{R}$ is non-negative, then the extension of $f$ by $0$ to all of $X$ is subharmonic on $X$.
\end{lma}
\begin{proof}
Suppose that $f\ge 0$ and define an extension of $f$ by $0$ as
\[\tilde{f}(x) = \begin{cases} f(x) & \colon x\in \partial X\\ 0 &\colon \text{ else.}\end{cases}\]

Now we observe that $\tilde{f}\in S(X)$. Suppose $x_0 \notin \partial X$.  Then $\tilde{f}(x_0)=0$, and as $\tilde{f}\ge 0$ everywhere, we have $\tilde{f}(x_0) \le \sum_{\zeta\in X} \lambda(x_0,\zeta)\tilde{f}(\zeta)$.
So $\tilde{f}$ is subharmonic at $x_0$.  If $x_1 \in \partial X$, by definition $\lambda(x_1,y)=\delta(x_1,y)$, the Kronecker delta.  Hence $\tilde{f}$ is harmonic at $x_1$.  Thus $\tilde{f}$ is subharmonic on $X$.
\end{proof}

\begin{thm}
For every non-negative bounded function $f\colon \partial X\rightarrow \mathbb{R}$ there is a bounded harmonic function $h\colon X\rightarrow \mathbb{R}$ such that $h=f$ on ${\partial X}$.  Furthermore $h$ is given by
\[h(x) = \lim_{k\rightarrow \infty} \omega_x^k(f)\]
for all $x\in X$.
\end{thm}
\begin{proof}
By Lemma \ref{L:ext_by_zero} the extension $\tilde{f}$ of $f$ by $0$ is subharmonic.  Therefore by Lemma \ref{L:mono} $\{\omega_x^k(\tilde{f})\}_{k=0}^\infty$ is an increasing sequence.  Since $f$ is bounded above, we have $f\le M$ for some positive constant $M$.  This shows that the sequence $\{\omega_x^k(\tilde{f})\}_{k=0}^\infty$ is bounded above by
\[\omega_x^k(\tilde{f}) = \sum_{\eta\in X} \lambda^{\diamond k}(x,\eta)\tilde{f}(\zeta)\le \sum_{\eta\in X} \lambda^{\diamond k}(x,\eta)M = M,\]
for all $x\in X$ and every $k$.
Thus this is a bounded increasing sequence and hence converges.  Let $h(x) = \lim_k \omega_x^k(\tilde{f})$.  Now we will check that $h$ is harmonic.

Since $\lambda(x,\zeta)$ is non-zero for only finitely many $\zeta$, we have
\begin{align*}\sum_{\zeta \in X}\lambda(x,\zeta)h(\zeta) & = \sum_{\zeta \in X}\lambda(x,\zeta)\lim_{k\rightarrow \infty} \omega_\zeta^k(\tilde{f}) \\& = \sum_{\zeta \in X}\lambda(x,\zeta)\lim_{k\rightarrow \infty} \left(\sum_{\eta\in X}\lambda^{\diamond k}(\zeta, \eta)\tilde{f}(\eta)\right) \\& = \lim_{k\rightarrow \infty} \sum_{\eta\in X}\left( \sum_{\zeta \in X}\lambda(x,\zeta)\lambda^{\diamond k}(\zeta, \eta)\right)\tilde{f}(\eta) \\& =
\lim_{k\rightarrow \infty} \sum_{\eta\in X}\lambda^{\diamond (k+1)}(x, \eta)\tilde{f}(\eta) \\& = \lim_{k\rightarrow \infty}\omega_x^{k+1}(\tilde{f}) = h(x).
\end{align*}
Hence $h$ is harmonic on $X$.

As $\tilde{f}(x)=f(x)$ for all $x\in \partial X$ and $\tilde{f}$ is harmonic at every point $x\in \partial X$ , we have that $\omega_x^{\diamond k}(\tilde{f})=f(x)$ for all $k$.  Hence $h=f$ on $\partial X$.
\end{proof}

For the remainder of the paper we will only consider graphs with finitely many ends, which we will denote $\omega$.  Modeling our work after classical potential theory, we must require the boundary data $f\colon \partial \hat{X}\rightarrow \mathbb{R}$ to be ``continuous at infinity".   Hence we make the following definitions:
\begin{df}
A sequence $\{x_n\}\subset \hat{X}$ is said to \emph{converge} to $x\in \hat{X}$, if for all $Y\subset X$ finite, there exists an $N>0$ such that for all $n>N$, $x_n$ and $x$ lie in the same connected component of $\hat{X}\setminus Y$.

Observe that the ends of the graph are the only limit points.  Hence by this definition a sequence can only converge if $x\in \partial_{end}X$.  We say that a function $f\colon \partial \hat{X}\rightarrow \mathbb{R}$ is \emph{continuous} at $\omega\in \partial_{end} X$, if $\lim_{x_n\rightarrow \omega}f(x_n)=f(\omega)$ for every $\{x_n\}\subset\partial\hat{X}$ converging to $\omega$.  If such an $f$ is continuous at every end, then it is said to be \emph{continuous at infinity}.
\end{df}

Note that in particular, if $\partial X$ is itself finite, then every function $f$ on $\partial X$ is necessarily continuous at infinity.  Furthermore that any harmonic function on $\hat{X}$ must necessarily be continuous at infinity as both $h$ and $-h$ are in $S(\hat{X})$.

Also note that if $X$ is a graph with one end $\omega$, then a function $f\colon \partial \hat{X}\rightarrow \mathbb{R}$ is {\it continuous at infinity}, whenever the set
\[\{x\in \partial X\colon |f(x)-f(\omega)|>\varepsilon\}\]
is finite for every $\varepsilon>0$.

In \cite{CSW93, W96} the solvability of the Dirichlet problem is characterized by two properties of the Green's function, which we will briefly recall.

The \emph{Green's function} on $X$ is defined by
\[G(x,y) = \sum_{k=0}^\infty \lambda^{\diamond k}(x,y), \]
for $x, y\in X$.  The Green's function is said to be \emph{transient} whenever $G(x,y)<\infty$ for all $x\in X^\circ$ and $y\in X$.  The Green's function is said to be \emph{vanishing at infinity} if
\[\lim_{n\rightarrow \infty}G(x_n,y)=0,\]
for every sequence $\{x_n\}\subset X$ which converges to some end. 

The condition that the Green's function be vanishing at infinity cannot be dropped. In \cite[Thm 7.15]{W94} it is stated that for graphs with at least two ends and a transient Green's function, the Dirichlet problem is solvable if and only if the Green's function vanishes at infinity.

The following lemma is easily seen to be a consequence of a transient Green's function which vanishes at infinity.
\begin{lma}\label{L:trans_G_at_inf}
Let $X$ be a connected quasi-reversible one ended graph admitting a transient Green's function which vanishes at infinity. Then for any point $y_0\in \partial X$ and any $\varepsilon>0$ there exists $N_1$ and $N_2$ such that for all $k>N_1$ and all points $x$ with $\dist(x,y_0)>N_2$, we have $\lambda^{\diamond k}(x,y_0)<\varepsilon$.  Here $N_1$ and $N_2$ are allowed to depend on the $y_0$ and $\varepsilon$.
\end{lma}

We are now ready to solve the Dirichlet problem for one ended graphs.
\begin{thm}\label{T:exists}
Let $X$ be a connected quasi-reversible one ended graph admitting a transient Green's function which vanishes at infinity.  If $f\colon \partial \hat{X}\rightarrow \mathbb{R}$ is bounded and continuous at infinity, then there exists a unique $h\in H(\hat{X})$ such that $h=f$ on $\partial\hat{X}$.
\end{thm}
\begin{proof}
Uniqueness follows from Theorem \ref{T:unique}.

Let $\omega$ denote the end of $X$.  By replacing $f$ with $f-f(\omega)$ we may assume without loss of generality that $f(\omega)=0$.

Given a function $f$ on $\partial\hat{X}$.  We consider functions $f^+ = \max\{f,0\}.$ and $f^- = -\min\{f,0\}$ so that $f=f^+-f^-$ and $|f|=f^++f^-$.  If $h^+$ and $h^-$ solve the problem for $f^+$ and $f^-$ respectively, then $h=h^+ - h^-$ solves the problem for $f$.  Furthermore if $h$ solves the problem for $f\ge 0$, then $-h$ solves the problem for $-f\le 0$.  Therefore it suffices to solve the problem with non-negative boundary data.  Hence we will assume for the remainder of the proof that $f$ is non-negative and $f(\omega)=0$.

Pick any $\varepsilon >0$.  By continuity at infinity the set
\[A=\{x\in \partial\hat{X} \colon f(x) \ge \varepsilon/2\}\]
is finite.  We will use the notation $|A|$ to denote the cardinality of $A$.

Let $M=\sup\{f(x)\colon x\in \partial\hat{X}\}$.  By Lemma \ref{L:trans_G_at_inf} we can find $N_1$ and $N_2$ such that for all $k>N_1$ and all $x\in X$ with $\dist(x,A)>N_2$ we have
\[\lambda^{\diamond k}(x,y)< \frac{\varepsilon}{2M|A|},\qquad (y\in A) \] where $\dist(x,A)$ denotes the maximum of $\dist(x,y)$ with $y\in A$.

As usual we denote the extension of $f$ by $0$ by $\tilde{f}$.  Hence for all $x\in X$ with $\dist(x,A)>N_2$ and any $k>N_1$, we have
\begin{align*}
\lambda^{\diamond k}(x,y)\tilde{f}(y) \le  \frac{\varepsilon}{2M|A|}M = \frac{\varepsilon}{2|A|} \qquad &(y\in A)\\
\lambda^{\diamond k}(x,y)\tilde{f}(y) \le  \lambda^{\diamond k}(x,y)\frac{\varepsilon}{2} \qquad &(y\not\in A)
\end{align*}

Thus for $k>N_1$ and all $x$ with $\dist(x,A)>N_2$ and any $k>N_1$, we have
\begin{align*}
\omega_x^k(\tilde{f}) &= \sum_{y\in X}\lambda^{\diamond k}(x,y)\tilde{f}(y) \\
&= \sum_{y\in X\setminus A}\lambda^{\diamond k}(x,y)\tilde{f}(y) +  \sum_{y\in A}\lambda^{\diamond k}(x,y)\tilde{f}(y) \\
&\le \sum_{y\in X\setminus A}\lambda^{\diamond k}(x,y)\frac{\varepsilon}{2} +  \sum_{y\in A}\frac{\varepsilon}{2|A|} \\
&\le \frac{\varepsilon}{2}\sum_{y\in X}\lambda^{\diamond k}(x,y) +  \frac{\varepsilon}{2} = \varepsilon
\end{align*}
Note that in the last step we used that $\sum_{y\in X}\lambda^{\diamond k}(x,y)=1$ for all $k$ and any $x\in X$.
As in the previous theorem, take $h(x)=\lim_{k\rightarrow \infty}\omega_x^k(\tilde{f})$ and observe that $h(x)\le \varepsilon$ whenever $\dist(x,A)>N_2$.

Since the graph is locally finite, the set
\[B=\{x\in X\colon \dist(x, A)\le N_2\}\]
is finite.  Notice that $h$ has the property that $h(x)\le \varepsilon$ for all $x\in X\setminus B$.  Now consider any ray $\pi=x_0x_1\ldots $ in the graph.  There can be at most finitely many vertices of $\pi$ in $X\setminus B$.  Thus
\[\limsup_{i\rightarrow \infty}h(x_i) \le \varepsilon\]
for every ray $\pi$ in the graph.

Since $f$ is assumed to be non-negative, it follows that $h$ too is non-negative.  As $\varepsilon$ was arbitrary we see that $\lim_{i\rightarrow \infty}h(x_i)=0$ along every ray $\pi=x_0x_1$ in the graph.  Hence $h$ extends harmonically to $\omega$ as $0$ which is, of course, the value of $f$ at $\omega$ as well.  In the previous theorem it was noted that $h=f$ on $\partial X$.  Thus $h\in H(\hat{X})$ and $h=f$ on $\partial \hat{X}$.
\end{proof}

We have now resolved the setting of the Dirichlet problem on the upper half plane (Example 1) which we originally used to motivate the need for including ends.  It is interesting to compare this result to that of \cite[Section 5]{K05}, in particular Example 5.2,  which also concerns the Dirichlet problem on the upper half plane with the $x$-axis as the boundary.  

Here one could extend the `boundary data' to the end by continuity, which is precisely what is occurring in Example 5.2 where one always gets a unique solution on the upper half plane.  By taking $f=0$ and $g=0$ in \cite[Example 5.2]{K05}, we are in a similar setting as Example 1 where we demonstrate that there are multiple solutions to the Dirichlet problem with this boundary condition.  The construction of \cite{K05} only produces the trivial solution $u=0$, implying that $g$ takes the value $0$ on the end. However in general there may be ends that are not limit points of $\partial X$.  

Furthermore the technique in Thm 5.1 of \cite{K05} requires a specific decomposition of the domain based on a growth condition.  As this decomposition is built inductively, it is not clear that one could include boundary values on the ends which are not limit points of $Y$ in a way in which they could be incorporated into the construction technique of \cite[Thm 5.1]{K05}.  Therefore the above Theorem \ref{T:exists} handles boundary data on an end which is not a limit point of $\partial X$ where \cite{K05} would not.

\section{Existence for graphs with finitely many ends.}
We will need one last definition before moving on to our final result.

\begin{df} A slice of a graph $X$ is a finite sequence $y_0y_1y_2\ldots y_m$ in $X$ such that consecutive vertices are adjacent, no vertex is repeated, and $y_0, y_m\subset \partial X$.
\end{df}

We can use the solution for the one-ended case together with a variant of the Schwartz Alternating Method \cite[pp 37--39]{GM08} to solve the multi-ended case.
\begin{thm}\label{T:exists2}
Let $X$ be a connected quasi-reversible graph with finitely many ends admitting a transient Green's function which vanishes at infinity.  If $f\colon \partial \hat{X}\rightarrow \mathbb{R}$ is bounded and continuous at infinity, then there exists a unique $h\in H(\hat{X})$ such that $h=f$ on $\partial\hat{X}$.
\end{thm}
\begin{proof}
Suppose $X$ has $n$ ends. We must now take $2n$ disjoint slices. For each end $\omega_i$, we take two disjoint slices $\gamma_{i, inside}$ and $\gamma_{i, outside}$, where each slice partitions $X$ into two components such that one of the components contains only one end $\omega_i$.  As the slices $\gamma_{i, inside}$ and $\gamma_{i, outside}$ are disjoint, we may suppose that $\gamma_{i, outside}$ is entirely contained  in the component of $X\setminus \gamma_{i, inside}$ that includes $\omega_i$.  We do this for each end.  

Let $X_i$ be the vertices of the component of $X\setminus \gamma_{i, inside}$ that includes $\omega_i$.  Let $Y_i = \{y\in X\colon \lambda(x,y) >0 \text{ for some } x\in X_i\}$.  We make $Y_i$ a graph by keeping the induced structure from $X$, except at the points $Y_i\setminus X_i$ which be come Kiselman boundary points, i.e. $\lambda(y,y)=1$ for all $y\in Y_i\setminus X_i$.  Let $X_{n+1}$ be the component of $X\setminus \left(\cup_{i=1}^n \gamma_{i, outside}\right)$ that contains no ends, and define $Y_{n+1} = \{y\in X \colon \lambda(x, y)>0 \text{ for some } x \in X_{n+1}\}$.  Again we give $Y_{n+1}$ the structure of a graph, by taking the induced structure on $X_{n+1}$ and making all $y\in Y_{n+1}\setminus X_{n+1}$ Kiselman boundary points.

To summarize, we've taking $X$ and grouped it into $n+1$ sections: $Y_1, \cdots Y_{n+1}$.  Each section $Y_i$ for $i=1$ to $n$ has exactly one end $\omega_i$, and $Y_{n+1}$ has no ends.  Furthermore the $Y_i$ for $i=1$ to $n$ are disjoint, but each overlaps with $Y_{n+1}$.  The new boundaries are defined so that the graph structure in the interior and on the overlapping sections remains unchanged. 

Consider our boundary data $f:\partial \hat{X}\rightarrow \mathbb{R}$.  On each piece $Y_i$ for $i=1$ to $n$ keep it the same only extend it to the new added boundary as $f(y)= \max\{f(x) \colon x \partial X\}$ for every $y\in Y_i\setminus X_i$, i.e. the new boundary points. Since each of these sections $Y_i$ has only one end we can find a harmonic extension $h_1$ of $f$ to $Y_i$ for for $i=1$ to $n$.  Without loss of generality we use the same labeling $h_1$ for the harmonic solution on each $Y_i$ for $i=1$ to $n$ as the are disjoint.  

Now we solve the Dirichlet problem on $Y_{n+1}$ with boundary data $f(x)$ when $x\in  X_{n+1}\cap \partial X$ or $h_1(y)$ when $y\in \partial Y_{n+1}\cap (X\setminus X_{n+1})$, that is we take the $h_1$ as the boundary values for the added boundary points of $Y_{n+1}$.  We call $h_1'$ the solution to the Dirichlet problem on $Y_{n+1}$ with this boundary data.  

Now we repeat this process, to find $h_2$ on $Y_i$ for $i=1$ to $n$ we use the boundary data $f$ on the original boundary, and $h_1'$ on the $\partial Y_i\setminus \partial X$.  Then we find $h_2'$ on $Y_{n+1}$ using the original boundary data where defined and $h_2$ on $\partial Y_{n+1}\setminus \partial X$.  And we repeat.

This gives us two sequences: $h_n$ and $h_n'$.  I claim they are related in the following way:
\begin{equation}\label{E:sal}
h_1\ge h_1'\ge h_2\ge h_2' \ge \cdots
\end{equation}

It follows from the maximum principle that $h_i\ge h_i'$ and that $h_i'\ge h_{i+1}$ as the boundary points of $Y_{n+1}$ are interior points of $Y_i$ for $i=1$ to $n$ and vice versa.

As $f$ is bounded below, the maximum principle implies that $h_i\ge \min\{f(x)\colon x\in \partial X\}$.  Hence the sequence \ref{E:sal} converges.  Let $h$ be the limit.  Since $h_i=f$ on $\partial X$ and $h_i'=f$ on $\partial X$ for every $i$, this means that the limit $h=f$ on $\partial X$.

It remains to be seen that $h$ is harmonic.  Let $x$ be an interior point of any of the $Y_i$ for $i=1$ to $n$.  Then as the sequence is decreasing 
\[h(x)\le h_i(x) = \sum\lambda(x,y)h_i(x) \]
for all $i$.  Hence 
\[h(x) \le \sum\lambda(x,y)h(x).\]

As $\lambda(x,y)\ge 0$, this implies that 
\[\sum\lambda(x,y)h(x) \le \sum\lambda(x,y)h_i(x)=h_i(x) \]
for all $i$. Hence 
\[\sum\lambda(x,y)h(x) \le h(x), \]
and $h$ is harmonic at $x$.  The same argument with $h_i'$ shows that $h$ is harmonic at the interior points of $Y_{n+1}$.  As the graph structure agrees on the overlap of the $Y_i$'s, this means that $h$ is harmonic everywhere. 
\end{proof}

\section*{Acknowledgments}
The author would like to thank Simon Smith for many helpful discussions.  We would also like to thank the reviewers for their many constructive and helpful suggestions.

\end{document}